\documentclass[reqno]{amsart}
\usepackage{amssymb,amsmath,amsfonts,amscd,amsthm,pb-diagram}
\usepackage{amsbsy,bm}

\newtheorem{theorem}{Theorem}[section]

\newtheorem{definition}{Definiton}[section]

\newtheorem{lemma}{Lemma}[section]
\newtheorem{proposition}{Proposition}[section]
\newtheorem{corollary}{Corollary}[section]
\theoremstyle{remark}
\newtheorem{remark}{Remark}[section]

\begin{document}
\title[Some rigidity phenomenons for Lagrangian submanifolds]
{On energy gap phenomena of the Whitney sphere and related problems}

\author{Yong Luo, Jiabin Yin}
\address{Mathematical Science Research Center, Chongqing University of Technology, Chongqing, 400054, China}
\email{yongluo-math@cqut.edu.cn}
\address{School of Mathematics and Statistics,Guangxi Normal University, Guilin, 541001, Guangxi, China}
\email{jiabinyin@126.com}

\thanks{The authors were supported by NSF of China, Grant Numbers 11501421, 11771404. The authors would like to thank Professor Zejun Hu for his useful suggestions which made this paper more readable. Many thanks to Dr. Xiuxiu Cheng for her interest and careful reading of the first version of this manuscript.}

\date{}


\begin{abstract}
In this paper, we study Lagrangian submanifolds satisfying ${\rm \nabla^*} T=0$ introduced by Zhang \cite{Zh} in the complex space forms $N(4c)(c=0\ or \ 1)$,
where $T ={\rm \nabla^*}\tilde{h}$ and $\tilde{h}$ is the Lagrangian trace-free second fundamental
form. We obtain some Simons' type integral inequalities and rigidity theorems for such  Lagrangian submanifolds. Moreover we study Lagrangian submanifolds in $\mathbb{C}^n$ satisfying $\nabla^*\nabla^*T=0$ and introduce a flow method related to them.
\end{abstract}

\maketitle

\numberwithin{equation}{section}
\section{Introduction}\label{sect:1}
In \cite{LW}, inspired by the study of Hamiltonian minimal submanifolds firstly introduced and studied by Oh \cite{Oh} and existence of minimizers of the  Willmore functional among compact Lagrangian tori in $\mathbb{C}^2$ proved by Minicozzi \cite{Mi}, Luo and Wang considered Lagrangian surfaces which are stationary points of the Willmore functional among Lagrangian surfaces or their Hamiltonian isotopy classes in $\mathbb{C}^2$, where such Lagrangian surfaces are briefly denoted by LW or HW surfaces respectively. Note that by their definition, the class of  HW surfaces is larger than that of LW surfaces, and for HW surfaces they proved the following energy gap theorem.
\begin{theorem}[\cite{LW}]
There exists a constant $\epsilon_0>0$ such that
if $\Sigma\hookrightarrow{\mathbb{C}^2}$ is a properly immersed HW surface with $\|h\|_{L^2(\Sigma)}<\epsilon_0$, where $h$ is the second fundamental form of the immersion,  then $\Sigma$ is a Lagrangian plane.
\end{theorem}
Gap phenomena for Willmore surfaces in a Euclidean space was firstly obtained by Kuwert and Sh\"atzler \cite{KS}, which is crucial in their proof of long time existence and convergence of the Willmore flow started from a compact surface with small energy. The study of gap theorems for HW surfaces was also motivated by long time behavior of the HW flow, a higher order flow defined by Luo and Wang \cite{LW}. But unfortunately, it seems hard to get certain energy gap theorems for compact HW surfaces by well known methods in the literature. It is natural to study the following
\\
\\\textbf{Problem 1:} Study certain energy gap phenomena for compact HW surfaces in $\mathbb{C}^2$.
\\

To explain what kind of energy gap theorems we may expect for compact HW surfaces we should mention several aspects on the Whitney spheres in $\mathbb{C}^n$, which are defined by
\begin{eqnarray*}
\phi_{r,A}: {\mathbb{S}^n}&{\to}& {\mathbb{C}^n}
 \\(x_1,...,x_{n+1})&\mapsto& \frac{r}{1+x_{n+1}^2}(x_1,x_1x_{n+1},...,x_n,x_nx_{n+1})+A,
\end{eqnarray*}
where $\mathbb{S}^n=\{(x_1,...,x_{n+1})\in {\mathbb{R}^{n+1}}|x_1^2+...+x_{n+1}^2=1\},$ $r$ is a positive number and $A$ is a vector of $\mathbb{C}^n$.
Since a compact Lagrangian sphere in $\mathbb{C}^2$ can not be embedded \cite{Gr}, by Li-Yau's inequality \cite{LY} for the Willmore energy of surfaces we see that the Willmore energy of Lagrangian spheres in $\mathbb{C}^2$ must be larger than or equal to $8\pi$, which is achieved by the (2-dimensional)Whitney spheres $\phi_{r,A}$. Carstro and Urbano \cite{CU} proved the Whitney spheres $\phi_{r,A}$ are the unique Lagrangian spheres with Willmore energy equal to $8\pi$, i.e. the Whitney spheres are the unique minimizers of the Willmore functional among Lagrangian spheres in $\mathbb{C}^2$. Note that the existence of minimizers of Willmore functional among Lagrangian tori in $\mathbb{C}^2$ was previously proved by Minicozzi \cite{Mi} mentioned above, by a direct variational method in spirit of Simon \cite{Sim}, where the assumption of embeddedness is necessary in their argument. Castro and Urbano \cite{CU2} also proved that the (2-dimensional)Whitney spheres $\phi_{r,A}$ are the unique Willmore Lagrangian spheres in $\mathbb{C}^2.$

The following 2-tensor, which is called the Lagrangian trace-free second fundamental form, plays an important role in the theory of Lagrangian submanifolds in $\mathbb{C}^n$(or more generally a complex space form).
\begin{eqnarray}
\tilde{h}(V,W):=h(V,W)-\frac{n}{n+2}\{\langle V, W\rangle H+\langle JV, H\rangle JW+\langle JW, H\rangle JV\},
\end{eqnarray}
where $h$ denotes the second fundamental form, $H=\frac{1}{n}h$ denotes the mean curvature vector field, $J$ and $\langle , \rangle$ denote the canonical complex structure and metric on $\mathbb{C}^n$ (or more generally a complex space form) respectively. Another important fact about the Whitney spheres $\phi_{r, A}$, proved by Castro and Urbano \cite{CU} for surfaces and Ros and Urbano \cite{RU} in general dimensions, is that the vanishing of $\tilde{h}$ for Lagrangian submanifolds in $\mathbb{C}^n$ implies that they must be an open part of totally geodesic Lagrangian submanifolds or the Whitney spheres $\phi_{r,A}$. This shows that the tensor $\tilde{h}$ plays a similar role with the trace-free second fundamental form of hypersurfaces in a Euclidean space and the Whitney spheres $\phi_{r,A}$ plays a role in $\mathbb{C}^n$ like that of round spheres in a Euclidean space.

Motivated by results and characterizations of the Whitney spheres $\phi_{r,A}$ mentioned above, one may naturally study the following
\\
\\\textbf{Problem 2:} Study certain energy gap phenomena for the Whitney spheres $\phi_{r,A}$.
\\

Partially motivated by problems 1 and 2 above, recently Zhang \cite{Zh} initiated the study of a new kind of Lagrangian surfaces in $\mathbb{C}^2$, satisfying the equation $\nabla^*T=0$, where $T$ is a (0, 2) tensor defined on $n$-dimensional Lagrangian submanifolds by (note that definitions of $T$ and $H$ in our paper differ from that of $T$ and $H$ in \cite{Zh} by a constant in the coefficients)
\begin{eqnarray}
T:=\frac{1}{n+2}\big(n\nabla (H\lrcorner \omega)-div(JH)g\big).
\end{eqnarray}
Zhang characterized properly immersed complete surfaces in $\mathbb{C}^2$ satisfying $\nabla^*T=0$ to be Lagrangian planes or the (2-dimensional)Whitney spheres $\phi_{r,A}$ if the $L^2$ energy of $\tilde{h}$ is small enough and the $L^2$ energy of $h$ decays like $o(\rho^2)$ on inverse image of balls of radius $\rho$ in $\mathbb{C}^2$ centered at the zero point.

Note that Lagrangian submanifolds in a complex space form with $T=0$ are called Lagrangian submanifolds with conformal Maslov form, which are Lagrangian counterparts of hypersurfaces in a real space form with constant mean curvature. This class of submanifolds in Lagrangian geometry was proposed and systematically studied in the 1990's by Castro, Montealegre, Ros, and Urbano \cite{CU, CMU, RU}. Later, Chao and Dong \cite{CD} obtained a Simons' type integral inequality for submanifolds with conformal Maslov form in a complex space form in $N^{n}(4c)(c=0\ or \ 1)$ and characterized the Whitney spheres as follows.
\begin{theorem}[\cite{CD}]\label{CD}
Let $M^n\hookrightarrow N^{n}(4c)(c=0\ or\ 1)(n\geq 2)$ be a compact (non-minimal) Lagrangian submanifold with conformal Maslov form (i.e. $T=0$). Then
$$\int_M^n|\tilde h|^2[|\tilde h|^2-\tfrac{4c(n+1)}{3(n+2)}-\tfrac{4n^2| H|^2}{3(n+2)^2}]d\nu\geq0.$$ If
\begin{eqnarray}\label{ass1}
|\tilde h|^2\leq\tfrac{4c(n+1)}{3(n+2)}+\tfrac{4n^2| H|^2}{3(n+2)^2},
\end{eqnarray}
then $M^n$ is the Whitney spheres $\phi_{r,A}$ when $c=0$ or the Whitney spheres $\phi_\theta$ when $c=1$.
\end{theorem}
Here the Whitney spheres $\phi_\theta$ in $N^n(4)=\mathbb{CP}^n$ are defined by (see (\cite{CMU}))
\begin{eqnarray*}
\phi_\theta: {\mathbb{S}^n}&{\to}& {\mathbb{CP}^n}, \theta>0
 \\(x_1,...,x_{n+1})&\mapsto& [(\frac{(x_1,...,x_n)}{ch\theta+ish\theta x_{n+1}}, \frac{sh\theta ch\theta(1+x^2_{n+1})+ix_{n+1}}{ch^2_\theta+sh^2_\theta x^2_{n+1}})].
\end{eqnarray*}
The Whitney spheres $\phi_\theta$ were characterized when $n=2$ by Castro and Urbano (\cite{CU1}) and general $n$ by Chen (\cite{Ch}) as the unique Lagrangian submanifolds in $\mathbb{CP}^n$ satisfying $\tilde{h}=0$. There are also Whitney spheres in $N^n(-4)=\mathbb{CH}^n$ and Lagrangian submanifolds in $\mathbb{CH}^n$ satisfying $\tilde{h}=0$ were also completely classified \cite{CU1}\cite{Ch}, where besides the totally geodesic submanifolds and the Whitney spheres, two new families of noncompact examples appeared.

In this paper, firstly we will improve and extend theorem \ref{CD} to Lagrangian submanifolds in complex space form $N^n(4c)(c=0 \ or\ 1)$ satisfying $\nabla^*T=0$.
\begin{theorem}\label{thm:1.1}
Let $M^n\hookrightarrow N^{n}(4c)(c=0\ or\ 1)(n\geq 3)$ be a compact Lagrangian submanifold with $\nabla^*T=0$. Then there holds the Simons' type integral inequality
\begin{equation}\label{eqn:1.03}
\int_M^n|\tilde h|^2\Big[|\tilde h|^2-\tfrac{2c(n+1)}{(n+3)}-\tfrac{2n^2}{(n+3)(n+2)}|H|^2\Big]d\nu\geq0.
\end{equation}
Moreover, equality in \eqref{eqn:1.03} holds if and only if $\tilde{h}=0$ and one of the following alternatives holds:
\begin{enumerate}
\item $c=0$, $M^n$ is the Whitney spheres $\phi_{r,A}$;
 \item $c=1$,   $M^n$ is either totally geodesic or the Whitney spheres $\phi_\theta$.
 \end{enumerate}
\end{theorem}
As a direct consequence of  theorem \ref{thm:1.1}, we have
\begin{corollary}\label{cor2}
Let $M^n\hookrightarrow N^{n}(4c)(c=0\ or\ 1)(n\geq 3)$ be a compact Lagrangian submanifold with $\nabla^*T=0$. If
\begin{eqnarray}\label{ass2}
|\tilde h|^2\leq\tfrac{2c(n+1)}{(n+3)}+\tfrac{2n^2}{(n+3)(n+2)}|H|^2,
\end{eqnarray}
then $\tilde h=0$, and one of the following alternatives holds:
\begin{enumerate}
\item $c=0$, $M^n$ is  the Whitney spheres $\phi_{r,A}$;
 \item $c=1$,   $M^n$ is either totally geodesic or the Whitney spheres $\phi_\theta$.
 \end{enumerate}
\end{corollary}

Note that our assumption of \eqref{ass2} is weaker than that of \eqref{ass1}, hence corollary \ref{cor2} is stronger than theorem \ref{CD} in the case $T=0$.
Moreover when $n=2$, we have the following optimal Simons' type integral inequalities.
\begin{theorem}\label{thm:1.2}
Let $\Sigma\hookrightarrow  N^2(4c)(c=0 \ or \ 1)$ be a compact Lagrangian surface with $\nabla^* T=0$. Then we have
\begin{equation*}
 \int_{\Sigma} | \tilde{h} |^2( |\tilde{ h} |^2-2- | H |^2)d\nu\geq0,
\end{equation*}
and equality holds if and only if one of the following alternatives holds:
\begin{enumerate}
\item $c=0$, $\Sigma$ is the (2-dimensional)the Whitney spheres $\phi_{r,A}$ or $\phi_{0,\alpha},\alpha\in (0,\pi]$;
\item $c=1$, $\Sigma$ is either totally geodesic , or the (2-dimensional)Whitney spheres $\phi_\theta$ or  the Calabi tori.
\end{enumerate}
\end{theorem}
As a direct consequence of the theorem \ref{thm:1.2}, we have
\begin{corollary}\label{cor1}
Let $\Sigma\hookrightarrow N^{2}(4c)(c=0 \ or \ 1)$ $(n\geq 2)$ be a compact Lagrangian submanifold with $\nabla^*T=0$. If
 $|\tilde h|^2\leq2c+|H|^2$, then one of the following alternatives holds:
\begin{enumerate}
\item $c=0$, $\Sigma$ is the (2-dimensional)Whitney spheres $\phi_{r,A}$ or $\phi_{0,\alpha},\alpha\in(0,\pi]$;
\item $c=1$, $\psi(\Sigma)$ is either totally geodesic , or the (2-dimensional)Whitney spheres $\phi_\theta$ or  the Calabi tori.
\end{enumerate}
\end{corollary}
\begin{remark}
$\phi_{0,\alpha}, \alpha\in (0,\pi]$ appearing in the statement of theorem \ref{thm:1.2} and corollary \ref{cor1} are examples of compact Lagrangian tori in $\mathbb{C}^2$ which were constructed and characterized by Castro and Urbano \cite{CU} to be the only compact flat Lagrangian tori in $\mathbb{C}^2$ with conformal Maslov form. In particular $\phi_{0,\pi}$ is the well known Clifford torus. The Calabi tori are compact Lagrangian tori in $\mathbb{CP}^2$ stated in appendix of \cite{LS}, which are characterized by Luo and Sun to be the only nonminimal compact Lagrangian surfaces in $\mathbb{CP}^2$ with parallel mean curvature vector field. Our classification result relies on their results.
\end{remark}
\begin{remark}
Similar integral inequality and characterization result were recently obtained by the first author \cite{Luo1,Luo2} for another kind of Lagrangian surfaces satisfying a higher order geometric partial differential equation, that is CSL surfaces in $\mathbb{S}^5$(or H-minimal surfaces in $\mathbb{CP}^2$). Here CSL  surfaces in $\mathbb{S}^5$ are Legendrian surfaces in $\mathbb{S}^5$ which are stationary points of the volume functional among Legendrian surfaces and H-minimal surfaces in $\mathbb{CP}^2$ are Lagrangian surfaces in $\mathbb{CP}^2$ which are stationary points of the volume functional among their Hamiltonian isotopy classes.
\end{remark}
Secondly, the energy gap theorem for Lagrangian surfaces in $\mathbb{C}^2$ satisfying $\nabla^*T=0$ due to Zhang was mainly inspired by the study of long time behavior of the HW flow defined by Luo and Wang \cite{LW}. It is also very natural to consider a flow related to Lagrangian submanifolds satisfying $\nabla^*T=0$. But unfortunately we can not find a proper way to achieve this. This  suggests us to consider a more general class of Lagrangian submanifolds in a symplectic manifold satisfying $\nabla^*\nabla^*T=0$. From the  point of view of partial differential equations, Lagrangian surfaces in $\mathbb{C}^2$ satisfying $\nabla^*\nabla^*T=0$ should have closer relationship with HW surfaces defined in \cite{LW}. Furthermore, unlike  HW surfaces, we can obtain an energy gap theorem for compact Lagrangian spheres in $\mathbb{C}^2$ satisfying $\nabla^*\nabla^*T=0$.
 \begin{theorem}\label{main thm3}
 Assume that $\Sigma\hookrightarrow\mathbb{C}^2$ is a Lagrangian sphere satisfying $\nabla^*\nabla^*T=0$. Then there exists a constant $\epsilon_0>0$ such that if
 \begin{eqnarray*}
 \int_{\Sigma}|\tilde{h}|^2d\nu\leq\epsilon_0,
 \end{eqnarray*}
 then $\Sigma$ is the (2-dimensional)Whitney spheres $\phi_{r,A}$.
 \end{theorem}

More importantly, we can introduce a flow method to study Lagrangian submanifolds in $\mathbb{C}^n$ satisfying $\nabla^*\nabla^*T=0$, and the well posedness of this flow can be verified similarly with that of HW flow in \cite{LW}.
\begin{theorem}\label{flow}
Let $M^n$ be a closed Lagrangian submanifold in $\mathbb{C}^n$. The flow
\begin{equation}\label{DCMF}
\left\{\begin{array}{rcl}
\frac{\partial M^n_t}{\partial_t} &=& -J\nabla\nabla^*\nabla^*T, \\
M^n_t|_{t=0}&=&M^n.
\end{array}\right.
\end{equation}
is well-posed. Namely, there exists a $T_0>0$ and an unique family of  Lagrangian surfaces $M^n_t$, $t\in [0, T_0)$, satisfying
\eqref{DCMF}.
\end{theorem}
The proof of theorem \ref{flow} follows the argument developed in the proof of the well posedness of the HW flow due to Luo and Wang \cite{LW}, by observing that such a flow could be rewritten in the cotangent bundle of its initial submanifold by using the famous Weinstein's tubular neighborhood theorem in symplectic geometry. Then we prove that it is equivalent to a parabolic six-order scalar flow.

We would like to point out that the well posedness of the flow \eqref{DCMF} can be proved when the target manifold is K\"ahler Einstein similarly with the $\mathbb{C}^n$ case, without extra difficulty.

The rest of this paper is organized as follows. In section 2 we give some preliminaries on Lagrangian submanifolds in a complex space form, including basic properties of the second fundamental form of Lagrangian submanifolds and the three basic equations(Gauss, Ricci, Codazzi) for Lagrangian submanifolds. In section 3 we prove our main theorem \ref{thm:1.1} and theorem \ref{thm:1.2}. In section 4 we study Lagrangian surfaces in $\mathbb{C}^n$ satisfying $\nabla^*\nabla^*T=0$ and a flow related to them. The proof of theorem \ref{main thm3} and theorem \ref{flow} will be sketched in this section. In the appendix we compute the quantity $\nabla^*\nabla^*T$ for Lagrangian graphs in detail, which helps to understand the proof of Theorem \ref{flow}.
\section{Preliminaries}\label{sect:2}~
Let $N^n(4c)$ be a complete, simply connected, $n$-dimensional K\"ahler manifold
with constant holomorphic sectional curvature $4c$. Let $M^n$ be an $n$-dimensional
Lagrangian submanifolds in $N^n(4c)$. We  denote also by $g$ the metric on
$M^n$. Let $\nabla$ (resp. $\bar\nabla$) be the Levi-Civita
connection of $M^n$ (resp. $N^{n}(4c)$). The Gauss and Weingarten
formulas of $M^n\hookrightarrow N^{n}(4c)$ are given,
respectively, by
\begin{equation}\label{eqn:2.1}
\bar\nabla_XY=\nabla_XY+h(X,Y)\ \ {\rm and}\ \
\bar\nabla_XV=-A_VX+\nabla^{\bot}_XV,
\end{equation}
where $X,Y\in TM^n$ are tangent vector fields, $V\in T^\perp M^n$ is a
normal vector field; $\nabla^{\bot}$ is the normal connection in the
normal bundle $T^\perp M^n$; $h$ is the second fundamental form and
$A_V$ is the shape operator with respect to $V$. From
\eqref{eqn:2.1}, we easily get
\begin{equation}\label{eqn:2.2}
g(h(X,Y),V)=g(A_VX,Y).
\end{equation}
The mean curvature vector $H$ of $M^n$ is defined by $H=\tfrac1n{\rm trace}\,h$.

For Lagrangian submanifolds, we have
\begin{equation}\label{eqn:2.3}
 \nabla_XJY=J\nabla^\bot_XY,
\end{equation}
\begin{equation}\label{eqn:2.4}
A_{JX}Y=-Jh(X,Y)=A_{JY}X.
\end{equation}
The above formulas immediately imply that $g(h(X,Y),JZ)$ is totally symmetric.

To utilize the moving frame method, we will use the following range
convention of indices:
\begin{equation*}
\begin{gathered}
i,j,k,l,m,p,s=1,\ldots,n; \ \
i^*=i+n \ etc..
\end{gathered}
\end{equation*}

Now, we choose a local {\it adapted Lagrangian frame} $\{e_1,\ldots, e_n,
e_{1^*},\ldots, e_{n^*}\}$ in $N^{n}(4c)$ in such a
way that, restricted  to $M^n$, $\{e_1,\ldots, e_n\}$ is an
orthonormal frame of $M^n$, and $\{e_{1^*}= Je_1,\ldots,e_{n^*}=J e_n\}$ is the orthonormal
normal vector fields of $M^n\hookrightarrow N^{n}(4c)$. Let
$\{\theta_1,\ldots,\theta_n\}$ be the dual frame of $\{e_1,\ldots ,
e_n\}$. Let $\theta_{ij}$ and $\theta_{i^*j^*}$ denote the
connection $1$-forms of $TM^n$ and $T^\perp M^n$, respectively.

Put $h_{ij}^{k^*} = g(h(e_i, e_j),J e_k)$. It is easily seen
that
\begin{equation}\label{eqn:2.5}
h_{ij}^{k^*}=h_{ik}^{j^*}=h_{jk}^{i^*},\ \ \forall\ i,j,k.
\end{equation}

Denote by $R_{ijkl}:=g\big(R(e_i,e_j)e_l,e_k\big)$ and
$R_{ijk^*l^*}:=g\big(R(e_i,e_j)e_{l*},e_{k^*}\big)$
the components of the curvature tensors of $\nabla$ and
$\nabla^{\bot}$ with respect to the adapted Lagrangian frame, respectively.
Then, we get the Gauss, Ricci and Codazzi equations:
\begin{equation}\label{eqn:2.6}
R_{ijkl}=c(\delta_{ik}\delta_{jl}-\delta_{il}\delta_{jk})
+\sum_{m}(h_{ik}^{m^*}h_{jl}^{m^*}-h_{il}^{m^*}h_{jk}^{m^*}),
\end{equation}
\begin{equation}\label{eqn:2.7}
R_{ijk^*l^*}=c(\delta_{ik}\delta_{jl}-\delta_{il}\delta_{jk})
+\sum_{m}(h_{ik}^{m^*}h_{jl}^{m^*}-h_{il}^{m^*}h_{jk}^{m^*}),
\end{equation}
\begin{equation}\label{eqn:2.8}
h^{m^*}_{ij,k}=h^{m^*}_{ik,j},
\end{equation}
where $h^{m^*}_{ij,k}$ is the components of the covariant
differentiation of $h$, defined by
\begin{equation}\label{eqn:2.9}
\sum_{l=1}^nh^{m^*}_{ij,l}\theta_l:=dh_{ij}^{m^*}+\sum_{l=1}^nh^{m^*}_{il}\theta_{lj}
+\sum_{l=1}^nh^{m^*}_{jl}\theta_{li}
+\sum_{l=1}^{n}h^{l^*}_{ij}\theta_{l^*m^*},
\end{equation}

Then from \eqref{eqn:2.5} and \eqref{eqn:2.8}, we have
\begin{equation}\label{eqn:2.10}
h^{m^*}_{ij,k}=h^{i^*}_{jk,m}=h^{j^*}_{km,i}=h^{k^*}_{mi,j}.
\end{equation}

We also have Ricci identity
\begin{equation}\label{eqn:2.11}
 h^{m^*}_{ij,lp}-h^{m^*}_{ij,pl}
=\sum_{k=1}^nh^{m^*}_{kj}R_{kilp}+\sum_{k=1}^nh^{m^*}_{ik}R_{kjlp}
+\sum_{k=1}^nh^{k^*}_{ij}R_{k^*m^{*}lp},
\end{equation}
where $h^{m^*}_{ij,lp}$ is defined by
\begin{equation*}
\sum_p
h^{m^*}_{ij,lp}\theta_p=dh^{m^*}_{ij,l}+\sum_ph^{m^*}_{pj,l}\theta_{pi}
+\sum_ph^{m^*}_{ip,l}\theta_{pj}+
\sum_ph^{m^*}_{ij,p}\theta_{pl}+\sum_{p}h^{p^*}_{ij,l}\theta_{p^*m^*}.
\end{equation*}
The mean curvature vector $H$ of
$M^n\hookrightarrow N^{n}(4c)$ is
$$
H=\tfrac1n\sum\limits_{i=1}^nh(e_i,e_i)=\sum\limits_{k=1}^nH^{k^*}e_{k^*},\
\ H^{k^*}=\tfrac1n\sum_ih^{k^*}_{ii}.
$$
Letting $i = j$ in \eqref{eqn:2.9} and carrying out summation over $i$, we have
\begin{equation*}
H^{k^*}_{,l}\theta_l=dH^{k^*}+\sum_{l}H^{l^*}\theta_{l^*k^*},
\end{equation*}
and we further have
\begin{eqnarray}\label{eqn:12}
{H}^{k^*}_{,i}=H^{i^*}_{,k}
\end{eqnarray} for any $i,k$.

\section{Lagrangian submanifolds in $N^n(4c)(c\geq0)$ satisfying $\nabla^*T=0$}\label{sect:3}
In this section, we will establish a Simons' type integral inequality
and gap theorem for Lagrangian submanifolds in $ N^{n}(4c)$ satisfying $\nabla^*T=0$, i.e. to give a proof of theorem \ref{thm:1.1} and theorem \ref{thm:1.2}.

In this section we assume that $M^n \hookrightarrow N^{n}(4c)(c\geq0)$ is a Lagrangian submanifold and $n\geq2$.

Firstly, we  define a trace-free tensor $\tilde h(X,Y)$ defined by
\begin{equation}\label{eqn:3.1}
\tilde h(X,Y)=h(X,Y)-\frac{n}{n+2}\big\{g(X,Y) H+g(J X,H)J Y+g(J Y,H)J X \big\}
\end{equation}
for any tangent vector fields $X,Y$ on $M^n$.

For $\tilde h=0$, Castro, Ros and Urbano for $c=0$ \cite{CU}\cite{RU}  and Castro, Chen and Urbano for $c\not=0$ \cite{CU1}\cite{Ch} proved the
following famous classification theorem:
\begin{theorem}
Let $x:M^n\rightarrow N^n(4c)(c=0\ or \ 1)$ be a Lagrangian submanifold, then
\begin{equation*}
  h(X,Y)=\frac{n}{n+2}\big\{g(X,Y) H+g(J X,H)J Y+g(J Y,H)J X \big\}
\end{equation*}
holds for any tangent vectors $X$ and $Y$ on $M^n$ if and only if $M^n$ is either an open part  of the
Whitney spheres or a totally geodesic submanifold.
\end{theorem}
With respect to Lagrangian frame
$\{e_1,\ldots, e_n,
e_{1^*},\ldots, e_{n^*}\}$ in $ N^n(4c)$, we have
\begin{equation}\label{eqn:3.2}
\begin{aligned}
\tilde h^{m^*}_{ij}=&h^{m^*}_{ij}-\frac{n}{n+2}\big(H^{m^*}\delta_{ij}+H^{i^*}\delta_{jm}
+H^{j^*}\delta_{im}\big)\\
=&h^{m^*}_{ij}-c^{m^*}_{ij}
\end{aligned}
\end{equation}
where $c^{m^*}_{ij}=\frac{n}{n+2}\big\{H^{m^*}\delta_{ij}+H^{i^*}\delta_{jm}
+H^{j^*}\delta_{im}\big\}$.

The first covariant derivatives of $\tilde h^{m^*}_{ij}$ are defined by
\begin{equation}\label{eqn:3.3}
\sum_{l=1}^n\tilde h^{m^*}_{ij,l}\theta_l:=d\tilde h_{ij}^{m^*}+\sum_{l=1}^n\tilde h^{m^*}_{il}\theta_{lj}
+\sum_{l=1}^n\tilde h^{m^*}_{jl}\theta_{li}
+\sum_{l=1}^{n}\tilde h^{l^*}_{ij}\theta_{l^*m^*}.
\end{equation}

The second covariant derivatives of $\tilde h^{m}_{ij}$ are defined by
\begin{equation}\label{eqn:3.5}
\sum_{l=1}^n\tilde h^{m^*}_{ij,kl}\theta_l:=d\tilde h_{ij,k}^{m^*}
+\sum_{l=1}^n\tilde h^{m^*}_{lj,k}\theta_{li}
+\sum_{l=1}^n\tilde h^{m^*}_{il,k}\theta_{lj}
+\sum_{l=1}^n\tilde h^{m^*}_{ij,l}\theta_{lk}
+\sum_{l=1}^{n}\tilde h^{\l^*}_{ij,k}\theta_{l^*m^*}.
\end{equation}

On the other hand, we have the following Ricci identities
\begin{equation}\label{eqn:3.5}
\tilde h^{m^*}_{ij,kp}-\tilde h^{m^*}_{ij,pk}
=\sum_{l}\tilde h^{m^*}_{lj}R_{likp}+\sum_l\tilde h^{m^*}_{il}R_{ljkp}
+\sum_{l}\tilde h^{l^*}_{ij}R_{l^*m^*kp}.
\end{equation}

The following proposition links those geometric quantities together:
\begin{lemma}
Let $M^n\hookrightarrow N^n(4c)$ be a Lagrangian submanifold, then the  Lagrangian trace-free second fundamental form $\tilde h$
satisfies
\begin{equation}\label{eqn:3.6}
 | \tilde h |^2= | h |^2-\frac{3n^2}{n+2} | H |^2.
\end{equation}
\begin{equation}\label{eqn:3.7}
\sum_m\tilde h^{m^*}_{ij,m}=\tfrac{n}{n+2}\big(nH^{i^*}_{,j}-{\rm div}JH\,g_{ij}\big).
\end{equation}
\end{lemma}
\begin{proof}
\eqref{eqn:3.6} and \eqref{eqn:3.7} can be immediately obtained from \eqref{eqn:3.2}.
\end{proof}

\begin{definition}
We define a $(0,2)$-tensor $T$ in local orthonormal basis as follows:
\begin{equation}\label{eqn:3.8}
T_{ij}=\tfrac1n\sum_m\tilde{h}^{m^*}_{ij,m}=\tfrac{1}{n+2}\big(nH^{i^*}_{,j}-\sum_mH^{m^*}_{,m}\,g_{ij}\big)
\end{equation}
\end{definition}
\begin{remark}
$T$ is a trace-free tensor and symmetric. $T=0$ if  and only if $JH$ is a conformal vector field.
\end{remark}
\begin{theorem}
Let  $M^n\hookrightarrow N^n(4c)$ be a Lagrangian submanifold. Then
\begin{equation*}
|\nabla JH|^2\geq\tfrac1n|{\rm div}JH|^2,
\end{equation*}
and equality  holds if and only $T=0$.
\end{theorem}
\begin{proof}
By direct computations and \eqref{eqn:3.8}, we have $0\leq|T|^2=(\tfrac{n}{n+2})^2\big(|\nabla JH|^2-\frac{1}{n}|{\rm div}JH|^2\big)$.
If $|\nabla JH|^2-\frac{1}{n}|{\rm div}JH|^2=0$, then $T=0$.
\end{proof}
\begin{remark}
In \cite{CU,CMU,RU}, Castro, Montealegre, Ros and Urbano systematically studied Lagrangian submanifolds in a complex space form with conformal Maslov
form (i.e. $T=0$) and obtained related classification theorems.
\end{remark}

In the following we will derive a Simons' type identity for $\Delta|\tilde h|^2$. First we have
\begin{lemma}\label{lem:3.2}
Let  $M^n\hookrightarrow N^n(4c)$ be a Lagrangian immersion. Then
\begin{equation}\label{eqn:3.9}
\begin{aligned}
\sum_{ijmk}\tilde h^{m^*}_{ij}\tilde h^{m^*}_{ij,kk}=&(n+2)\sum_{m,i,j}(\tilde h^{m^*}_{ij}T_{ij})_{,m}
-{n^2}\sum_{ij}\big[( H^{i^*}T_{ij})_{,j}-H^{i^*}T_{ij,j}\big]\\
&+ \sum_{i,j,m,k,l}\tilde h^{m^*}_{ij}
\Big(\tilde h^{m^*}_{lk}R_{lijk}
+\tilde h^{m^*}_{il}R_{lkjk}
+\tilde h^{l^*}_{ik}R_{l^*m^*jk}\Big)
\end{aligned}
\end{equation}
\end{lemma}

\begin{proof}
By using the Codazzi equation \eqref{eqn:2.8} and \eqref{eqn:3.2}, the definition of $\tilde h$ under local coordinates is just
\begin{equation}\label{eqn:3.10}
 \tilde h^{m^*}_{ij,k}=\tilde h^{m^*}_{ik,j}+\tfrac{n}{n+2}\big(\delta_{ik}H^{m^*}_{,j}
 +\delta_{km}H^{i^*}_{,j}-\delta_{ij}H^{m^*}_{,k}-\delta_{jm}H^{i^*}_{,k}\big)
\end{equation}

With the help of Ricci identity \eqref{eqn:3.5}, \eqref{eqn:3.8} and \eqref{eqn:3.10}, we have
\begin{equation}\label{eqn:3.11}
\begin{aligned}
 \sum_k \tilde h^{m^*}_{ij,kk}=&\sum_k\tilde h^{m^*}_{ik,jk}+\sum_k\tfrac{n}{n+2}\big(\delta_{ik}H^{m^*}_{,jk}
 +\delta_{km}H^{i}_{,jk}-\delta_{ij}H^{m^*}_{,kk}-\delta_{jm}H^{i^*}_{,kk}\big)\\
 =&\sum_k \tilde h^{m^*}_{ik,kj}+
\sum_{k,l}\tilde h^{m^*}_{lk}R_{lijk}+\sum_{k,l}\tilde h^{m^*}_{il}R_{lkjk}
+\sum_{k,l}\tilde h^{l^*}_{ik}R_{l^*m^*jk}\\
&+\sum_k\tfrac{n}{n+2}\big(\delta_{ik}H^{m^*}_{,jk}
 +\delta_{km}H^{i}_{,jk}-\delta_{ij}H^{m^*}_{,kk}-\delta_{jm}H^{i^*}_{,kk}\big)\\
 =&\sum_k \tilde h^{m^*}_{kk,ij}+
\sum_{k,l}\tilde h^{m^*}_{lk}R_{lijk}+\sum_{k,l}\tilde h^{m^*}_{il}R_{lkjk}
+\sum_{k,l}\tilde h^{l^*}_{ik}R_{l^*m^*jk}\\
&+\sum_k\tfrac{n}{n+2}\big(\delta_{ik}H^{m^*}_{,jk}
 +\delta_{km}H^{i^*}_{,jk}-\delta_{ij}H^{m^*}_{,kk}-\delta_{jm}H^{i^*}_{,kk}\big)+nT_{im,j}.
 \end{aligned}
\end{equation}
Then, by using \eqref{eqn:3.8} and the fact that $\tilde{h}$ is trace free and tri-symmetric, we have
\begin{equation*}
\begin{aligned}
\sum_{i,j,m,k}\tilde h^{m^*}_{ij}\tilde h^{m^*}_{ij,kk}=&\sum_{i,j,m,k}\tilde h^{m^*}_{ij}\Big[\tilde h^{m^*}_{lk}R_{lijk}+\tilde h^{m^*}_{il}R_{lkjk}+\tilde h^{l^*}_{ik}R_{l^*m^*jk}\Big]\\
&+\sum_{m,i,j}\tilde h^{m^*}_{ij}\Big[T_{mj,i}+T_{ij,m}\Big]+n\sum_{m,i,j}\tilde h^{m^*}_{ij}T_{im,j}\\
=&\sum_{i,j,m,k}\tilde h^{m^*}_{ij}\Big[\tilde h^{m^*}_{lk}R_{lijk}+\tilde h^{m^*}_{il}R_{lkjk}+\tilde h^{l^*}_{ik}R_{l^*m^*jk}\Big]\\
&+(n+2)\sum_{m,i,j}\tilde h^{m^*}_{ij}T_{ij,m}
\end{aligned}
\end{equation*}

In addition, we have
\begin{equation*}
\begin{aligned}
\sum_{mij}\tilde h^{m^*}_{ij}T_{ij,m}=&\sum_{m,i,j}(\tilde h^{m^*}_{ij}T_{ij})_{,m}-\sum_{m,i,j}\tilde h^{m^*}_{ij,m}T_{ij}
\\=&\sum_{m,i,j}(\tilde h^{m^*}_{ij}T_{ij})_{,m}-n \sum_{i,j}T_{ij}T_{ij}\\
=&\sum_{m,i,j}(\tilde h^{m^*}_{ij}T_{ij})_{,m}-n\sum_{i,j}T_{ij}\tfrac{1}{n+2}\big(nH^{i^*}_{,j}-\sum_mH^{m^*}_{,m}\,g_{ij}\big)
\\=&\sum_{m,i,j}(\tilde h^{m^*}_{ij}T_{ij})_{,m}-\tfrac{n^2}{n+2}\sum_{i,j} H^{i^*}_{,j}T_{ij}\\
=&\sum_{m,i,j}(\tilde h^{m^*}_{ij}T_{ij})_{,m}-\tfrac{n^2}{n+2}\sum_{ij}\big[( H^{i^*}T_{ij})_{,j}-H^{i^*}T_{ij,j}\big],
\end{aligned}
\end{equation*}
where in the third equality we used \eqref{eqn:3.8} and the fact that $trace_gT=0$. Thus, we obtain the assertion.
\end{proof}
Next, by using lemma \ref{lem:3.2}
\begin{equation}\label{eqn:3.12}
\begin{aligned}
  \frac12\Delta | \tilde h|^2
=&|\nabla \tilde h|^2+\sum_{ijmk}\tilde h^{m^*}_{ij}\tilde h^{m^*}_{ij,kk}
\\=&|\nabla \tilde h|^2+(n+2)\sum_{i,j,m}(\tilde h^{m^*}_{ij}T_{ij})_{,m}-{n^2}\sum_{i,j}\big[( H^{i^*}T_{ij})_{,j}-H^{i^*}T_{ij,j}\big]\\
&+\underbrace{\sum_{i,j,k,m,l}\tilde h^{m^*}_{ij}\tilde h^{m^*}_{lk}R_{lijk}}_{I}
+\underbrace{\sum_{i,j,k,m,l}\tilde h^{m^*}_{ij}\tilde h^{m^*}_{il}R_{lkjk}}_{II}
+\underbrace{\sum_{i,j,k,m,l}\tilde h^{m^*}_{ij}\tilde h^{l^*}_{ik}R_{lmjk}}_{III}.
\end{aligned}
\end{equation}
Note that by the symmetry of $\tilde{h}_{ij}^{k^*}$, $I=III$. Hence we only need to compute $I$ and $II$. Direct computations show that
\begin{equation}\label{eqn:3.13}
\begin{aligned}
I=&c\sum_{i,j,k,m,l}\tilde h^{m^*}_{ij}\tilde h^{m^*}_{kl}(\delta_{lj}\delta_{ik}-\delta_{lk}
  \delta_{ij})+\sum_{i,j,k,m,l,t}\tilde h^{m^*}_{ij}\tilde h^{m^*}_{kl}(\tilde h^{t^*}_{lj}\tilde h^{t^*}_{ik}-\tilde h^{t^*}_{lk}
  \tilde h^{t^*}_{ij})\\
  &+\sum_{i,j,k,m,l,t}\tilde h^{m^*}_{ij}\tilde h^{m^*}_{kl}(\tilde h^{t^*}_{lj}c^{t^*}_{ik}+c^{t^*}_{lj}\tilde h^{t^*}_{ik}-\tilde h^{t^*}_{lk}
  c^{t^*}_{ij}-c^{t^*}_{lk}\tilde h^{t^*}_{ij}+c^{t^*}_{lj}c^{t^*}_{ik}-c^{t^*}_{lk}c^{t^*}_{ij})\\
  =& c| \tilde h |^2+\tfrac{n^2}{(n+2)^2} | \tilde h |^2 | H |^2+\tfrac{2n}{n+2}\sum_{j,k,l,m,t}\tilde h^{m^*}_{jk}
  \tilde h^{m^*}_{kl}\tilde h^{t^*}_{lj}H^{t^*}\\
  &+\sum_{i,j,k,m,l,t}\tilde h^{m^*}_{ij}\tilde h^{m^*}_{kl}(\tilde h^{t^*}_{lj}\tilde h^{t^*}_{ik}-\tilde h^{t^*}_{lk}
  \tilde h^{t^*}_{ij})+\tfrac{2n^2}{(n+2)^2}\sum_{i,j,k,m}\tilde h^{m^*}_{ij}\tilde h^{m^*}_{jk}H^{i^*}H^{k^*},
 \end{aligned}
\end{equation}
where in the second equality we used the following identities derived by direct computations
\begin{eqnarray*}
\sum_{i,j,k,m,l,t}\tilde h^{m^*}_{ij}\tilde h^{m^*}_{kl}\tilde h^{t^*}_{lj}c^{t^*}_{ik}=\sum_{i,j,k,m,l,t}\tilde h^{m^*}_{ij}\tilde h^{m^*}_{kl}c^{t^*}_{lj}\tilde h^{t^*}_{ik}=\tfrac{3n}{n+2}\sum_{j,k,l,m,t}\tilde h^{m^*}_{jk}
  \tilde h^{m^*}_{kl}\tilde h^{t^*}_{lj}H^{t^*},
\end{eqnarray*}
\begin{eqnarray*}
\sum_{i,j,k,m,l,t}\tilde h^{m^*}_{ij}\tilde h^{m^*}_{kl}\tilde h^{t^*}_{lk}
  c^{t^*}_{ij}=\sum_{i,j,k,m,l,t}\tilde h^{m^*}_{ij}\tilde h^{m^*}_{kl}c^{t^*}_{lk}\tilde h^{t^*}_{ij}=\tfrac{2n}{n+2}\sum_{j,k,l,m,t}\tilde h^{m^*}_{jk}
  \tilde h^{m^*}_{kl}\tilde h^{t^*}_{lj}H^{t^*},
\end{eqnarray*}
and since
\begin{equation*}
\begin{aligned}
\sum_tc^{t^*}_{lj}c^{t^*}_{ik}=&\frac{n^2}{(n+2)^2}\sum_t\Big((\mathfrak{S}_{t,l,j}
H^{t^*}\delta_{lj})(\mathfrak{S}_{t,i,k}H^{t^*}\delta_{ik})
\Big)\\
=&\frac{n^2}{(n+2)^2}\Big(\mathfrak{S}_{i,j,k,l}H^{l^*}H^{i^*}\delta_{jk}+2H^{l^*}H^{j^*}\delta_{ik}\\
&+2H^{i^*}H^{k^*}\delta_{jl}+|H|^2\delta_{ik}\delta_{jl}\Big),
\end{aligned}
\end{equation*}
where $\mathfrak{S}$  stands for the cyclic sum, thus
\begin{eqnarray*}
\sum_{i,j,k,m,l,t}\tilde h^{m^*}_{ij}\tilde h^{m^*}_{kl}c^{t^*}_{lj}c^{t^*}_{ik}=\tfrac{n^2}{(n+2)^2} | \tilde h |^2 | H |^2+\tfrac{6n^2}{(n+2)^2}\sum_{i,j,k,m}\tilde h^{m^*}_{ij}\tilde h^{m^*}_{jk}H^{i^*}H^{k^*},
\end{eqnarray*}
\begin{eqnarray*}
\sum_{i,j,k,m,l,t}\tilde h^{m^*}_{ij}\tilde h^{m^*}_{kl}c^{t^*}_{lk}c^{t^*}_{ij}=\tfrac{4n^2}{(n+2)^2}\sum_{i,j,k,m}\tilde h^{m^*}_{ij}\tilde h^{m^*}_{jk}H^{i^*}H^{k^*}.
\end{eqnarray*}
Similarly we have
\begin{equation}\label{eqn:3.14}
\begin{aligned}
II&=\sum_{i,j,k,m,l}\tilde h^{m^*}_{ij}\tilde h^{m^*}_{il}[c(\delta_{lj}\delta_{kk}-\delta_{lk}\delta_{jk})+\sum_t(h_{lj}^{t^*}h_{kk}^{t^*}-h_{lk}^{t^*}h_{kj}^{t^*})
\\&=(n-1)c | \tilde h|^2+\sum_{i,j,k,m,l,t}\tilde h^{m^*}_{ij}\tilde h^{m^*}_{li}\big(n\tilde h^{t^*}_{lj}H^{t^*}+nc^{t^*}_{lj}H^{t^*}\\
&-\tilde h^{t^*}_{lk}
  \tilde h^{t^*}_{kj}-\tilde h^{t^*}_{lk}c^{t^*}_{kj}-c^{t^*}_{lk}\tilde h^{t^*}_{kj}-c^{t^*}_{lk}c^{t^*}_{kj}\big)\\
&=(n-1)c | \tilde h |^2+\tfrac{n^3}{(n+2)^2} | \tilde h |^2 | H |^2+\tfrac{n^2-2n}{n+2}\sum_{ijlmt}\tilde h^{m^*}_{ij}
  \tilde h^{m^*}_{li}\tilde h^{t^*}_{lj}H^{t^*}\\
  &+\tfrac{n^2(n-2)}{(n+2)^2}\sum_{ijml}\tilde h^{m^*}_{ij}\tilde h^{m^*}_{li}H^{j^*}H^{l^*}-\sum_{i,j,k,m,l,t}\tilde h^{m^*}_{ij}\tilde h^{m^*}_{li}\tilde h^{t^*}_{lk}\tilde h^{t^*}_{kj},
 \end{aligned}
\end{equation}
where in the third equality we used the following identities derived by direct computations
\begin{eqnarray*}
n\sum_{i,j,k,m,l,t}\tilde h^{m^*}_{ij}\tilde h^{m^*}_{li}c^{t^*}_{lj}H^{t^*}=\frac{n^2}{n+2} | \tilde h |^2 | H |^2+\frac{2n^2}{n+2}\sum_{ijml}\tilde h^{m^*}_{ij}\tilde h^{m^*}_{li}H^{j^*}H^{l^*},
\end{eqnarray*}
\begin{eqnarray*}
\sum_{i,j,k,m,l,t}\tilde h^{m^*}_{ij}\tilde h^{m^*}_{li}\tilde h^{t^*}_{lk}c^{t^*}_{kj}=\sum_{i,j,k,m,l,t}\tilde h^{m^*}_{ij}\tilde h^{m^*}_{li}\tilde h^{t^*}_{kj}c^{t^*}_{lk}=\frac{2n}{n+2}\sum_{ijlmt}\tilde h^{m^*}_{ij}
  \tilde h^{m^*}_{li}\tilde h^{t^*}_{lj}H^{t^*},
\end{eqnarray*}
and since
\begin{equation*}
\begin{aligned}
\sum_tc^{t^*}_{lk}c^{t^*}_{jk}=&\frac{n^2}{(n+2)^2}\sum_t\Big((\mathfrak{S}_{t,l,k}
H^{t^*}\delta_{lk})(\mathfrak{S}_{t,j,k}H^{t^*}\delta_{jk})
\Big)\\
=&\frac{n^2}{(n+2)^2}\Big(\mathfrak{S}_{j,k,k,l}H^{l^*}H^{j^*}\delta_{kk}+2H^{l^*}H^{k^*}\delta_{jk}\\
&+2H^{j^*}H^{k^*}\delta_{kl}+|H|^2\delta_{jk}\delta_{kl}\Big),
\end{aligned}
\end{equation*}
thus
\begin{eqnarray*}
\sum_{i,j,k,m,l,t}\tilde h^{m^*}_{ij}\tilde h^{m^*}_{li}c^{t^*}_{lk}c^{t^*}_{kj}=\frac{2n^2}{(n+2)^2} | \tilde h |^2 | H |^2+\frac{(n+6)n^2}{(n+2)^2}\sum_{ijml}\tilde h^{m^*}_{ij}\tilde h^{m^*}_{li}H^{j^*}H^{l^*}.
\end{eqnarray*}
\begin{equation}\label{eqn:3.15}
\begin{aligned}
III=I=& c| \tilde h |^2+\tfrac{n^2}{(n+2)^2} | \tilde h |^2 | H |^2+\tfrac{2n}{n+2}\sum_{j,k,l,m,t}\tilde h^{m^*}_{jk}
  \tilde h^{m^*}_{kl}\tilde h^{t^*}_{lj}H^{t^*}\\
  &+\sum_{i,j,k,m,l,t}\tilde h^{m^*}_{ij}\tilde h^{m^*}_{kl}(\tilde h^{t^*}_{lj}\tilde h^{t^*}_{ik}-\tilde h^{t^*}_{lk}
  \tilde h^{t^*}_{ij})\\
  &+\tfrac{2n^2}{(n+2)^2}\sum_{i,j,k,m}\tilde h^{m^*}_{ij}\tilde h^{m^*}_{jk}H^{i^*}H^{k^*}.
 \end{aligned}
\end{equation}

Set $A_{i*}=(\tilde h^{i^*}_{jk})$. Then it follows from \eqref{eqn:3.9} and \eqref{eqn:3.12}-\eqref{eqn:3.15} that
\begin{equation}\label{eqn:3.16}
\begin{aligned}
\frac12\Delta | \tilde h |^2=&(n+2)\sum_{i,j,m}(\tilde h^{m^*}_{ij}T_{ij})_{,m}-{n^2}\sum_{i,j}\big[( H^{i^*}T_{ij})_{,j}-H^{i^*}T_{ij,j}\big]\\
&+|\nabla\tilde h|^2+(n+1) c| \tilde h |^2
  +\tfrac{n^2}{(n+2)} | \tilde h |^2 | H |^2\\
  &+\sum_{i,j}{\rm tr}(A_{i^*}A_{j^*}-A_{j^*}A_{i^*})^2-\sum_{i,j}({\rm tr}A_{i^*}A_{j^*})^2\\
  &+n\sum_{i,j,l,m,t}\tilde h^{m^*}_{ji}\tilde h^{m^*}_{jt}\tilde h^{l^*}_{ti}H^{l^*}
  +\tfrac{n^2}{(n+2)}\sum_{i,j,k,m}\tilde h^{m^*}_{ij}\tilde h^{m^*}_{jk}H^{i^*}H^{k^*}
 \end{aligned}
\end{equation}

In order to prove the theorem \ref{thm:1.1} and \ref{thm:1.2}, we need the following lemma.
\begin{lemma}[\cite{LiLi}]\label{lem:3.3}
Let $B_1,\ldots,B_m$ be symmetric $(n\times n)$-matrices $(m\geq2)$. Denote
$S_{mk}={\rm trace}(B^t_{m}B_{k})$, $S_{m}=S_{mm}=N(B_{m})$,
$S=\sum_{i=1}^mS_{i}$. Then
\begin{equation*}
\sum_{m,k}N(B_{m}B_{k}-B_{k}B_{m})+\sum_{m,k}S^2_{mk}\leq\frac32S^2.
\end{equation*}
\end{lemma}

Now we are prepared to estimate the right hand side of (\ref{eqn:3.16}), mainly the last two terms on the last line of (\ref{eqn:3.16}).

 By re-choosing $\{e_i\}_{i=1}^n$ such that $\sum_l\tilde h^{l^*}_{ij}H^{l^*}=\lambda_i\delta_{ij}$, denoting $$S_{H}=\sum_{j,i}(\sum_l\tilde h^{l^*}_{ji}H^{l^*})^2=\sum_{j}\lambda^2_j,\ \ S_{i^*}=\sum_{j,l}(\tilde h^{i^*}_{jl})^2,$$  and $S=\sum_iS_{i^*}$.
Then by using lemma \ref{lem:3.3} and the above fact, we have the following estimation for the right hand side of (\ref{eqn:3.16})
\begin{equation}\label{eqn:3.17}
\begin{aligned}
\frac12\Delta | \tilde h |^2\geq&(n+2)\sum_{i,j,m}(\tilde h^{m^*}_{ij}T_{ij})_{,m}-{n^2}\sum_{i,j}\big[( H^{i^*}T_{ij})_{,j}-H^{i^*}T_{ij,j}\big]\\
&+|\nabla\tilde h|^2+(n+1)c | \tilde h |^2
  +\tfrac{n^2}{(n+2)} | \tilde h |^2 | H |^2-\tfrac32S^2\\
  &+n\sum_i\lambda_iS_{i^*}+\tfrac{n^2}{n+2}\sum_i\lambda^2_{i}\\
  \geq&(n+2)\sum_{i,j,m}(\tilde h^{m^*}_{ij}T_{ij})_{,m}-{n^2}\sum_{i,j}\big[( H^{i^*}T_{ij})_{,j}-H^{i^*}T_{ij,j}\big]\\
&+|\nabla\tilde h|^2+(n+1)c | \tilde h |^2
  +\tfrac{n^2}{(n+2)} | \tilde h |^2 | H |^2-\tfrac32S^2\\
  &+\tfrac{n}2\sum_i(\lambda_i+S_{i^*})^2-\tfrac{n}2\sum_iS^2_{i^*}\\
  \geq&(n+2)\sum_{i,j,m}(\tilde h^{m^*}_{ij}T_{ij})_{,m}-{n^2}\sum_{i,j}\big[( H^{i^*}T_{ij})_{,j}-H^{i^*}T_{ij,j}\big]\\
&+|\nabla\tilde h|^2+(n+1)c | \tilde h |^2
  +\tfrac{n^2}{(n+2)} | \tilde h |^2 | H |^2-\tfrac{n+3}2S^2\\
  &+\tfrac{n}2\sum_i(|H|\lambda_i+S_{i^*})^2,\\
 \end{aligned}
\end{equation}
where, we have used $S^2=(\sum_iS_{i^*})^2\geq\sum_iS_{i^*}^2$.

Since $M^n$ is compact and $\nabla^*T=0$, we have
\begin{equation*}
 0\geq\int_{M^n}|\tilde h|^2\Big[(n+1)c
  +\tfrac{n^2}{(n+2)} | H |^2-\tfrac{n+3}2|\tilde h|^2\Big]d\nu.
\end{equation*}
Therefore, equality holds in the above inequality implies $M^n$ is totally geodesic or $\tilde h=0$(here we used the fact that equality in lemma \ref{lem:3.3} can not be achieved except $M^n$ is totally geodesic if $n\geq3$, by Li and Li \cite{LiLi}). Then, from classification results of
 Ros and Urbano \cite{RU} and Chen \cite{Ch} we have one of the following alternatives holds:
\begin{enumerate}
 \item $c=0$, $M^n$ is the Whitney spheres $\phi_{r,A}$;
 \item $c=1$,   $M^n$ is either totally geodesic or the Whitney spheres $\phi_\theta$.
\end{enumerate}
This completes the proof of theorem \ref{thm:1.1}.
\\
\\

When $n=2$,  we obtain from \eqref{eqn:3.12} that

\begin{equation}\label{eqn:3.18}
\begin{aligned}
0=& \int_M|\nabla \tilde h|^2+3K|\tilde h|^2d\nu
\\&\geq \int_M3K|\tilde h|^2d\nu
\\&=\int_M\tfrac32(2c+|H|^2-|\tilde h|^2)|\tilde h|^2d\nu,
\end{aligned}
\end{equation}
where $K$ is the Gauss curvature of $M$. Here the last equality holds since by \eqref{eqn:2.6} and \eqref{eqn:3.6}  we have
\begin{eqnarray*}
K&=&c+detA_{1^*}+detA_{2^*}
\\&=&c+\frac{1}{2}(4|H^{1^*}|^2-|A_{1^*}|^2)+\frac{1}{2}(4|H^{2^*}|^2-|A_{2^*}|^2)
\\&=&c+\frac{1}{2}(4|H|^2-|h|^2)
\\&=&c+\frac{1}{2}(|H|^2-|\tilde{h}|^2).
\end{eqnarray*}
 Therefore, equality  holds in \eqref{eqn:3.18} if and only if $\tilde h=0$ or
$\nabla\tilde h=0$ and $K=0$.
\\

\textbf{Case 1.} $\tilde{h}=0$.  From classification results of
Carstro and Urbano \cite{CU} and Chen\cite{Ch}, we have one of the following alternatives holds:
\begin{enumerate}
 \item $c=0$, $\Sigma$ is the (2-dimensional)Whitney spheres $\phi_{r,A}$;
 \item $c=1$,   $\Sigma$ is either totally geodesic or the (2-dimensional)Whitney spheres $\phi_\theta$.
\end{enumerate}

\textbf{Case 2.} $\nabla\tilde h=0$ and $K=0$. Since $\nabla^*T=0$, it is easy to see that
\begin{equation*}
\begin{aligned}
0=&\tfrac{1}{4}(2\sum_jH_{,jj}^{i^*}-\sum_{k,j}H_{,kj}^{k^*}g_{ij})\\
=&\tfrac{1}{4}(H^{i^*}_{,jj}+KH^{i^*}).
\end{aligned}
\end{equation*}
Then,
\begin{equation}\label{eqn:3.19}
0=\tfrac12\int_M\Delta |H|^2d\nu=\int_M|\nabla JH|^2-K|H|^2d\nu=\int_M|\nabla JH|^2d\nu,
\end{equation}
which implies that $\nabla JH=0$. Therefore  $\Sigma$ is flat Lagrangian surface with $\nabla\tilde h=0$ and $\nabla JH=0$.

If $c=1$,  from proposition 3.3 of \cite{LS} for the nonminimal case and the discussion on page 853 of \cite{Ha} for minimal case, we see that  $\Sigma$ is a flat Calabi torus stated in the appendix
of \cite{LS}.

If $c=0$, from proposition 3 and remark 6 of \cite{CU} we see that $\Sigma$ is a Lagrangian immersion given by $\phi_{0,\alpha}(\alpha\in(0,\pi])$ , defined in \cite{CU}.

 This completes the proof of theorem \ref{thm:1.2}.

\section{Lagrangian surfaces in $\mathbb{C}^n$ satisfying $\nabla^*\nabla^*T=0$}
In this section we study Lagrangian surfaces in $\mathbb{C}^n$ satisfying equation $\nabla^*\nabla^*T=0$, as well as a flow method related with them. The aim of this section is to prove theorem \ref{main thm3} and theorem \ref{flow}.
\subsection{Proof of Theorem \ref{main thm3}}
 From the proof of theorem 1.1 in \cite{Zh}, we see that for any cut off function $\gamma\in C^1_c(\Sigma)$, with $|\nabla\gamma|\leq\frac{C_0}{R}$, if $\epsilon_0$ is small enough, we have
\begin{eqnarray}
\int_\Sigma(|\nabla\tilde{h}|^2+|H|^2|\tilde{h}|^2)\gamma^2d\nu\leq C\int_\Sigma\langle\nabla^*T, H\lrcorner \omega\rangle\gamma^2d\nu+\frac{C}{R^2}\int_{\gamma>0}|h|^2d\nu.
\end{eqnarray}
Since $\Sigma$ is compact, when $R$ is large enough we have $\gamma\equiv1$ on $\Sigma$, therefore for $R$ sufficiently large
$$\int_\Sigma(|\nabla\tilde{h}|^2+|H|^2|\tilde{h}|^2)d\nu\leq C\int_\Sigma\langle\nabla^*T, H\lrcorner \omega\rangle d\nu+\frac{C}{R^2}\int_\Sigma|h|^2d\nu.$$
Now because $\Sigma$ is topologically a sphere and $H\lrcorner\omega$ is a closed one form on $\Sigma$ by Dazord \cite{Da}, there exists a smooth function $f$ on $\Sigma$ such that $H\lrcorner \omega=df$. Hence
$$\int_\Sigma\langle\nabla^*T, H\lrcorner \omega\rangle d\nu=\int_\Sigma\langle\nabla^*\nabla^*T, f\rangle d\nu=0.$$
We obtain
$$\int_\Sigma(|\nabla\tilde{h}|^2+|H|^2|\tilde{h}|^2)d\nu\leq\frac{C}{R^2}\int_\Sigma|h|^2d\nu.$$
Letting $R\to \infty$, we get $\tilde{h}=0$, and therefore $\Sigma$ is the Whitney sphere, by Castro and Urbano \cite{CU}.

\subsection{Proof of Theorem \ref{flow}}The proof of theorem \ref{flow} could be completed by following the argument used in  proof of the well posedness of the HW flow in \cite{LW}. For the convenience of readers we gives some details here. Note that $$T_{ij}=\frac{1}{n+2}(nH_{,j}^{i^*}-\sum_kH_{,k}^{k^*}g_{ij}),$$ hence \begin{eqnarray*}
(\nabla^*T)_i&=&\frac{1}{n+2}(n\sum_jH_{,jj}^{i^*}-\sum_{k,j}H_{,kj}^{k^*}g_{ij})
\\&=&\frac{1}{n+2}(n\sum_jH_{,jj}^{i^*}-\sum_{k,j}(H_{,jk}^{k^*}g_{ij}-\rm Ric^j_{\ k}H^{k^*}g_{ij}))
\\&=&\frac{1}{n+2}(n\sum_jH_{,jj}^{i^*}-\sum_{k,j}H_{,kk}^{j^*}g_{ij}+\sum_{k,j}\rm Ric_{ik}H^{k^*})
\\&=&\frac{n-1}{n+2}\sum_jH_{,jj}^{i^*}+\sum_{k,j}\frac{1}{n+2}\rm Ric_{ik}H^{k^*},
\end{eqnarray*}
where in the second equality we used Ricci formula to switch the second order covariant derivatives of $H$ and in the last equality we used the symmetry of $H^{k^*}_{,j}$ by (\ref{eqn:12}). Therefore we have
$$\nabla^*T=\frac{n-1}{n+2}(\Delta^\bot H)\lrcorner\omega+\frac{1}{n+2}\rm Ric(JH),$$
 and
$$\nabla^*\nabla^*T=\frac{n-1}{n+2}\rm{div}\Delta(JH)+\frac{1}{n+2}\nabla^*Ric(JH).$$
Then if $M^n_t$ is a graph, i.e. there exists $\varphi_t\in C^\infty(M^n)$ such that
$$M^n_t=(x_1,x_2,...,x_n,\partial_1\varphi_t,\partial_2\varphi_t,...,\partial_2\varphi_t),$$ where $(x_1,x_2,...,x_n)$ are local coordinates on $M^n$ and $\partial_i\varphi_t=\frac{\partial\varphi_t}{\partial x_i}$, then the leading term of $\nabla^*\nabla^*T$ for $M^n_t$ is $\frac{n-1}{n+2}\Delta^3\varphi_t$ (see the appendix for some details). Then by similar discussion with that of \cite{LW}, on page 216, the flow equation \eqref{DCMF} is reduced to a parabolic scalar equation (see the appendix for some details)
\begin{eqnarray*}
\frac{d}{dt}\varphi_t=\frac{n-1}{n+2}\Delta^3\varphi_t+\ lower \ order\ terms.
\end{eqnarray*}

In general, by the argument in \cite{LW}, the flow equation \eqref{DCMF} can be seen as a flow equation for Lagrangian surfaces in the cotangent bundle of the initial surface, by using Weinstein's tubular neighborhood theorem. Then the flow equation can be seem as a parabolic scalar equation in a short time interval, similar with the graph case discussed above. At last the well posedness of the flow follows  from a general theorem of Huisken and Polden \cite{HP} for the well posedness of scalar parabolic partial differential equations on a Riemannian manifold.

\section{Appendix}
In this appendix we discuss the equation for Lagrangian graphs in $\mathbb{C}^n$ satisfying $\nabla^*\nabla^*T=0$. Note that we have shown in the last section that
\begin{eqnarray}
\nabla^*\nabla^*T=\frac{n-1}{n+2}\rm div\Delta(JH)+\frac{1}{n+2}\nabla^*Ric(JH).
\end{eqnarray}
Assume that $l$ is a Lagrangian graph, that is there is a function $\varphi:\mathbb{R}^n{\to} \mathbb{R}$ such that
 $$l(x_1,x_2,...,x_n)=(x_1,x_2,...,x_n,\partial_1\varphi,\partial_2\varphi,...,\partial_n\varphi).$$

Assume that $\theta$ is the Lagrangian angle of $l$, then it is well known that \cite{Mor}
\begin{eqnarray}
J\nabla\theta=H,
\end{eqnarray}
and since by Ricci formula
\begin{eqnarray*}
\Delta\nabla\theta=\nabla\Delta\theta+\rm \sum_{j,k}Ric^k_{\ j}\theta_ke_j.
\end{eqnarray*}
We have
\begin{eqnarray*}
\rm div\Delta(JH)&=&\rm -div\Delta\nabla\theta
\\&=&-\Delta^2\theta-\rm div \sum_{j,k}Ric^k_{\ j}\theta_ke_j.
\end{eqnarray*}
Therefore
\begin{eqnarray}
\nabla^*\nabla^*T=-\frac{n-1}{n+2}\Delta^2\theta-\frac{2}{n+2}\rm div \sum_{j,k}Ric^k_{\ j}\theta_ke_j.
\end{eqnarray}
Note that $\rm div \sum_{j,k}Ric^k_{\ j}\theta_ke_j$ contains at highest fourth order terms of $\varphi$ and hence the leading term of $\nabla^*\nabla^*T$ is contained in $-\frac{n-1}{n+2}\Delta^2\theta$, which we will compute in the following.

Note that by definition the Lagrangian angle, now globally defined, is
\begin{eqnarray}
e^{i\theta}=*( dz_1\wedge dz_2\wedge\cdot\cdot\cdot\wedge dz_n)\mid_{l(\mathbb{R}^n)},
\end{eqnarray}
where $*$ is the Hodge star operator with respect to the metric $g_l$.
Thus we have
\begin{eqnarray*}
e^{i\theta}&=&*(dx_1+id\varphi_1)\wedge(dx_2+id\varphi_2)\wedge\cdot\cdot\cdot\wedge(dx_n+id\varphi_n)\\
&=& \frac{1}{\sqrt{\det(g_l)}} \det (I_n+i D^2_0 \varphi),
\end{eqnarray*}
where $\Delta_0=\frac{\partial^2}{\partial x_1^2}+\frac{\partial^2}{\partial x_2^2}+\cdot\cdot\cdot+\frac{\partial^2}{\partial x_n^2}$ is the standard Laplacian  on $\mathbb{R}^n$, $D^2_0\varphi$ is the Hessian of $\varphi$
and $I_n$ is the $n\times n$ identity matrix.

A direct computation gives
\[\det(g_l)=\det(I_n+(D^2_0\varphi)^2).\]
Therefore we have
\begin{equation}\label{add_c}
\theta=-i \log \frac{\det (I_n+i D^2_0 \varphi)}{\sqrt{\det(I_n+(D^2_0\varphi)^2)}}
\end{equation}

Let \begin{eqnarray*}
&&A=(A_{ij})=I_n+i D^2_0\varphi,\ B=(B_{ij})=g_l=I_n+(D^2_0\varphi)^2,
\\ &&A^{-1}=(A^{ij}), \ B^{-1}=(B^{ij})=g_l^{-1}.
\end{eqnarray*}
The matrices $A$ and $B$ have the following relation
\[B=(I_n-iD^2_0 \varphi )A,\quad A^{-1}=B^{-1}  (I_n-iD^2_0 \varphi).\]
Using the above relation, we have (following \cite{CCY})
\[
\begin{array}{rcl}
\theta_k  &=&  -i\sum_{i,j} (A^{ij} A_{ij,k}- \frac12  B^{ij}  B_{ij,k} )\\
&=& -i\sum_{i,j} (B^{il}(\delta_{lj}-i\varphi_{lj}) \cdot i \varphi_{ijk}-\frac 12 B^{ij} \cdot 2 \varphi_{il}\varphi_{ljk})\\
&=&\sum_{i,j} B^{ij}\varphi_{ijk}
\\&=&\sum_{i,j}g^{ij}\varphi_{ijk}.
\end{array}
\]
Therefore
\begin{eqnarray}
\Delta\theta=\Delta^2\varphi+lot_1,
\end{eqnarray}
where $lot_1$ contains at highest third order terms of $\varphi$ and hence
\begin{eqnarray}
-\Delta^2\theta=-\Delta^3\varphi+lot_2,
\end{eqnarray}
where $lot_2$ contains at highest fifth order terms of $\varphi$. In conclusion,
\begin{proposition}
Assume that $l: \mathbb{R}^n\to \mathbb{C}^n(n\geq2)$ is a Lagrangian graph with $l(x_1,x_2,\cdot\cdot\cdot,x_n)=(x_1,x_2,\cdot\cdot\cdot,x_n,\partial_1\varphi,\partial_2\varphi,\cdot\cdot\cdot,\partial_n\varphi)$ satisfying $\nabla^*\nabla^*T=0$. Then $l$ satisfies the following sixth order quasi-linear elliptic equation of $\varphi$
\begin{eqnarray}
\nabla^*\nabla^*T=-\frac{n-1}{n+2}\Delta^2\theta-\frac{2}{n+2}\rm div  \sum_{j,k}Ric^k_{\ j}\theta_ke_j=0,
\end{eqnarray}
where $\theta$ is the Lagrangian angle of $l$. In particular, $l$ satisfies the following equation
$$\nabla^*\nabla^*T=-\frac{n-1}{n+2}\Delta^3\varphi+lot_3=0,$$
where $lot_3$ contains at highest fifth order terms of $\varphi$.
\end{proposition}

\end{document}